\documentclass[10pt,reqno]{amsart}

\usepackage[numbers]{natbib}
\usepackage[utf8]{inputenc}
\usepackage{amsmath}
\usepackage{amsfonts}
\usepackage{amssymb}
\usepackage{mathrsfs}
\usepackage{mathtools}
\usepackage{enumitem}
\usepackage{url}
\usepackage{hyperref}
\usepackage[noabbrev,capitalize,nameinlink]{cleveref}

\usepackage{amsthm}
\newtheorem{thm}{Theorem}[section]
\newtheorem{lem}[thm]{Lemma}
\newtheorem{prop}[thm]{Proposition}

\theoremstyle{definition}
\newtheorem{defn}[thm]{Definition}

\theoremstyle{remark}
\newtheorem*{rem}{Remark}

\newcommand{\overbar}[1]{\mkern 1.5mu\overline{\mkern-1.5mu#1\mkern-1.5mu}\mkern 1.5mu}
\newcommand\abs[1]{\left|#1\right|}

\newcommand\loc{\mathrm{loc}}

\DeclareMathOperator{\Div}{div}

\DeclareMathOperator{\supp}{supp}

\DeclareMathOperator{\dist}{dist}

\numberwithin{equation}{section}

\begin{document}

\title[Stability]{Rigidity and Stability of Submanifolds with Entropy Close to One}

\author{Letian Chen}
\address{Department of Mathematics, Johns Hopkins University, 3400 N. Charles Street, Baltimore, MD 21218}
% \curraddr{}
\email{lchen155@jhu.edu}
% \thanks{}

%\subjclass[2010]{Primary }
\date{Mar. 15, 2020}
% \dedicatory{}

\begin{abstract}
We show that an $n$-dimensional surface whose entropy is close to that of an $n$-dimensional plane is close in Hausdorff distance to some $n$-dimensional plane at every scale. Moreover we show that self-expanders of low entropy converge in the Hausdorff sense to their asymptotic cones.
\end{abstract}

\maketitle

\section{Introduction}
Given an $n$-dimensional properly embedded (connected) smooth submanifold of $\mathbb{R}^{n+k}$, Colding-Minicozzi \cite{CM1} introduced the entropy of $\Sigma$:
\begin{align*}
    \lambda[\Sigma] = \sup_{x_0 \in \mathbb{R}^{n+k}, t_0 > 0} F[t_0\Sigma + x_0] = \sup_{x_0,t_0} \int_{\Sigma} \frac{1}{(4\pi t_0)^{n/2}} e^{-\frac{\abs{x-x_0}^2}{4t_0}} d\mathcal{H}^n
\end{align*}
where 
\begin{align*}
    F[\Sigma] = (4\pi)^{-n/2} \int_\Sigma e^{-\frac{\abs{x}^2}{4}} d\mathcal{H}^n
\end{align*}
is the Gaussian surface area of $\Sigma$. The constant $(4\pi)^{-n/2}$ is to ensure that the entropy of an $n$-dimensional plane is 1. Note that by definition the entropy is invariant under ambient scaling and translation. The entropy plays a key role in the study of mean curvature flow since, by Huisken's monotonicity formula \cite{Hui}, it is a nonincreasing quantity under the flow. \par 
Many works have been done concerning closed hypersurfaces of low entropy, in particular spheres and (generalized) cylinders. Bernstein and L. Wang \cite{BW1} have shown that for $2 \le n \le 6$ and $k=1$, the round sphere $\mathbb{S}^n$ minimizes the entropy among all closed hypersurfaces in $\mathbb{R}^{n+1}$ and is rigid in the sense that $\lambda[\Sigma] = \lambda[\mathbb{S}^n]$ implies $\Sigma = \mathbb{S}^n$ up to a translation and rotation. Later Zhu \cite{Zhu} generalized this to all dimensions. In this paper we examine the case for noncompact submanifolds. It is obvious that $n$-planes minimize the entropy among all smooth embedded surfaces. Our first result shows that the planes are rigid, i.e. they are the only surfaces with entropy 1. 
\begin{thm}
\label{t14}
Suppose $\Sigma$ is an $n$-dimensional embedded submanifold of $\mathbb{R}^{n+k}$ (possibly with boundary). If $\lambda[\Sigma] = 1$, then $\Sigma$ is flat. That is, the second fundamental form vanishes at every point $x \in \Sigma$. In particular, if $\Sigma$ is properly embedded in $\mathbb{R}^{n+k}$, then $\Sigma$ is an $n$-dimensional plane.
\end{thm}
We note that, in the case $k = 1$ (i.e. hypersurfaces), rigidity is straightforward if we have some additional bound on the second fundamental form on $\Sigma$. Indeed, by short time existence of Ecker and Huisken \cite{EH}, we can run the mean curvature flow on $\Sigma$. By Huisken's monotonicity formula it follows $\Sigma$ is a self-shrinker of entropy 1, which must then be a plane (cf. Proposition 2.10 in \cite{Whi1}). Thus the main issue here is that without a curvature bound the surface may not be able to flow at all. \par 
To get by this issue, we will use mean curvature flow with boundary inside an ambient ball $\overbar{B_R(0)}$ in $\mathbb{R}^{n+k}$, the theory of which has recently been developed by White \cite{Whi2}. When we use a flow in a compact set, short time existence is always guaranteed (in fact by White's work the flow exists in the weak sense for all positive time). Using a slight variant of White's local regularity theorem (Corollary 3.4 in \cite{Whi1}) we can conclude that $\Sigma$ is flat in $B_{R/4}(0)$. Since $R$ is arbitrary we conclude that $\Sigma$ is a plane. \par  
Bernstein and L. Wang \cite{BW2} classified self-shrinkers of low entropy in $\mathbb{R}^3$ and were able to prove a Hausdorff stability result in \cite{BW3}:
\begin{thm}[Bernstein-Wang]
\label{t12}
Given $\varepsilon > 0$, there is $\delta = \delta(\varepsilon) >0$ such that if $\Sigma$ is a closed hypersurface in $\mathbb{R}^3$ with $\lambda[\Sigma] < \lambda[\mathbb{S}^2] + \delta$, then 
\begin{align*}
    \dist_H(\Sigma,\rho\mathbb{S}^2 + y) < \rho \varepsilon
\end{align*}
for some $\rho > 0$ and $y \in \mathbb{R}^3$.
\end{thm}
This is subsequently generalized by S. Wang \cite{Wang} to all higher dimensions using a forward clearing out lemma (as there is currently no classification of low entropy self-shrinker available in higher dimensions). We prove the corresponding theorem for $n$-planes. That is, given a smooth, properly embedded $\Sigma$ with entropy sufficiently close to 1, then $\Sigma$ is Hausdorff close to some $n$-plane up to a translation. Here the closeness is interpreted in the sense of Reifenberg \cite{Rei} (see also \cite{Her}): 
\begin{defn}
\label{d13}
    Given an $n$-dimensional smooth, properly embedded submanifold $\Sigma$ of $\mathbb{R}^{n+k}$, the (Reifenberg) planar distance of $\Sigma$ is 
    \begin{align*}
        \dist_P(\Sigma) = \sup_{p \in \Sigma} \sup_{R > 0} \inf_{P} \frac{1}{R}\dist_H((P+p) \cap B_R(p), \Sigma \cap B_R(p))
    \end{align*}
    where $\dist_H(\cdot,\cdot)$ is the Hausdorff distance and the infimum is taken over all $n$-dimensional planes $P$ passing through the origin in $\mathbb{R}^{n+k}$.
\end{defn}
Evidently this distance is scaling and translation invariant. Our main theorem is the following:
\begin{thm}
\label{t13}
Given $\varepsilon > 0$, there is $\delta = \delta(\varepsilon,n,k) > 0$ such that for every $n$-dimensional, noncompact smooth, properly embedded submanifold $\Sigma$ of $\mathbb{R}^{n+k}$ with $\lambda[\Sigma] < 1 + \delta$ we have $\dist_P(\Sigma) < \varepsilon$. 
\end{thm}
\begin{rem}
In view of Reifenberg's original definition this is saying that any $\Sigma$ with $\lambda[\Sigma] < 1 + \delta$ is $(\varepsilon,R)$-Reifenberg flat for every $R > 0$.
\end{rem}
The assumption that $\Sigma$ is noncompact is redundant. In fact, there is $\delta_0 > 0$ depending on $n$ and $k$ only such that for any closed $n$-dimensional smooth surface $\Sigma$ in $\mathbb{R}^{n+k}$ one has $\lambda[\Sigma] \ge 1 + \delta_0$. In the case $k=1$ the optimal constant is $\delta_0 = \lambda[\mathbb{S}^n] - 1$ by the works \cite{BW1} and \cite {Zhu}. In arbitrary codimension the existence of such a $\delta_0$ is a consequence of White's local regularity theorem \cite{Whi1}: if $\lambda[\Sigma] < 1 + \delta_0$ then the mean curvature flow never develops a singularity by virtue of the curvature bound, but the mean curvature flow of a closed surface must develop a singularity in finite time by parabolic maximum principle (see for example \cite{Eck}). \par 
To prove \cref{t13} we adapt the scaling argument used in the proof of \cref{t12}. Briefly, we argue by contradiction and take a sequence of smooth surfaces with entropy bounded above by $1 + \frac{1}{n}$. We run a mean curvature flow with boundary for each one of them and after normalization (possibly passing to a subsequence) these flows converge to a flow of entropy 1, which by \cref{t13} has to be a flow of $n$-planes. $C^1_{\loc}$ convergence then implies Hausdorff closeness which gives a contradiction. \par 
Finally we show that self-expanders of low entropy are close to a fixed cone in the Hausdorff sense. More precisely, recall that an $n$-dimensional smooth, properly embedded surface $\Sigma$ in $\mathbb{R}^{n+k}$ is a self-expander if it is a solution to the equation 
\begin{align*}
    \frac{x^\perp}{2} = H_{\Sigma}(x)
\end{align*}
It follows that $\{\sqrt{t}\Sigma\}_{t > 0}$ is a smooth mean curvature flow in $\mathbb{R}^{n+k}$. Self-expanders serve as singularity models of the mean curvature flow as they give rise to potential continuation of the flow past a (conical) singularity. We refer to \cite{BW4} and \cite{Ding} and references therein for more about (asymptotically conical) self-expanders. Given a self-expander $\Sigma$, there is a natural way to associate a (unique) cone, $\mathcal{C}_\Sigma$, by looking at its space-time track under the mean curvature flow (see \cref{exp}). There is no guarantee that this cone is regular, so one generally cannot expect closeness in $C^{k,}$. Our result says that at least in the low entropy case, self-expanders are close to this cone in (pointed) Hausdorff topology. Moreover this cone is actually quite "large" in the sense that it is flat in the Reifenberg sense.
\begin{thm}
\label{t15}
There is $\delta = \delta(n,k) > 0$ such that for every $n$-dimensional smooth, properly embedded self-expander $\Sigma \subseteq \mathbb{R}^{n+k}$ with $\lambda[\Sigma] < 1 + \delta$ there is a (unique) cone $\mathcal{C}$ such that 
\begin{align*}
    \lim_{t \to 0} \dist_H(\sqrt{t}\Sigma \cap \overbar{B_R(0)},\mathcal{C} \cap \overbar{B_R(0)}) = 0
\end{align*}
for all $R > 0$. In fact, there is some constant $C = C(n,k)$ such that
\begin{align*}
    \dist_H(\sqrt{t}\Sigma \cap \overbar{B_R(0)},\mathcal{C} \cap \overbar{B_R(0)}) \le C \sqrt{t}
\end{align*}
for sufficiently small $t$. 
\end{thm}
\begin{rem}
In codimension 1 one can use instead Lemma 6.1 of \cite{Wang} to relax the entropy bound to $\lambda[\Sigma] < \lambda[\mathbb{S}^n] + \delta$. 
\end{rem}
\subsection*{Acknowledgements} The author is extremely grateful to his advisor, Jacob Bernstein, for suggesting the problem and for helpful advice and constant encouragement. The author also wishes to thank Lu Wang and Shengwen Wang for helpful conversations, and Brian White for some comments on mean curvature flow with boundary.
\section{Preliminaries}
\subsection{Notation} 
For $x \in \mathbb{R}^{n+k}$ and $t \in \mathbb{R}$, $B_r(x)$ denotes the open ball in $\mathbb{R}^{n+k}$ and $C_r(x,t)$ denotes the parabolic cylinder given by 
\begin{align*}
    C_R(x,t) = B_R(x) \times (t-R^2,t+R^2)
\end{align*}
in $\mathbb{R}^{n+k} \times \mathbb{R}$. \par 
Given an open set $U \subseteq \mathbb{R}^{n+k}$, a mean curvature flow $\mathcal{M}$ in $U$ is a collection of smooth $n$-dimensional submanifolds $\{\Sigma_t\}_{t \in I}$ for some interval $I$ such that 
\begin{enumerate}
    \item For every $t_0 \in I$ and $x_0 \in \Sigma_t$ there is $r > 0$ such that $C_R(x_0,t_0) \subseteq U \times I$. 
    \item There is a smooth map $\Psi: B_1(0) \times (t_0 - R^2,t_0+R^2) \to \mathbb{R}^{n+k}$ such that $\Psi_t(x) = \Psi(x,t): B_1(0) \to \mathbb{R}^{n+k}$ is a parameterization of $B_R(x) \cap \Sigma_t$.
    \item The parametrization $\Psi$ satisfies the mean curvature equation 
    \begin{align*}
        \left(\frac{\partial \Psi(x,t)}{\partial t}\right)^{\perp} = H_{\Sigma_t}(\Psi(x,t))
    \end{align*}
    where $H_{\Sigma_t}$ is the mean curvature vector of $\Sigma_t$.
\end{enumerate}
We sometimes also use $\mathcal{M}$ to denote the space-time track of the flow, namely 
\begin{align*}
    \mathcal{M} = \{(x,t) \in U \times I \mid x \in \Sigma_t\}
\end{align*} \par 
When $U = \mathbb{R}^{n+k}$, we define the entropy of the mean curvature flow to be 
\begin{align*}
    \lambda[\mathcal{M}] = \sup_{t \in I} \lambda[\Sigma_t]
\end{align*}
which by Huisken's monotonicity formula is equal to the entropy of the initial surface. The definition of Gaussian surface area and entropy extends naturally to surfaces in $U$, and thus to smooth mean curvature flow in $U$ (see for example \cite{BW3}). We note here that if $\Sigma$ is a smooth submanifold of $\mathbb{R}^{n+k}$, then $\lambda[\Sigma \cap U] \le \lambda[\Sigma]$. \par 
Given a mean curvature flow $\mathcal{M} = \{\Sigma_t\}_{t \in [t_0,t_1]}$ and a space-time point $(x,t) \in U \times [t_0,t_1]$, the Gaussian density ratio of $\mathcal{M}$ at $(x_0,t) \in U \times (t_0,t_1]$ of scale $r$ is defined as 
\begin{align*}
    \Theta(\mathcal{M},(x_0,t),r) = \int_{\Sigma_{t-r^2}} \frac{1}{(4\pi r^2)^{n/2}} e^{-\frac{\abs{x-x_0}^2}{4r^2}} d\mathcal{H}^n
\end{align*}
provided $t-r^2 > t_0$. The Gaussian density at $(x_0,t)$ is the limit of the above expression as $r$ decreases to 0, which exists by Huisken's monotonicity formula.
\subsection{Mean Curvature Flow with Boundary}
Here we summarize the work of White \cite{Whi2} concerning mean curvature flow with boundary. Although White's original work deals only with the case of hypersurfaces, the existence result generalizes essentially verbatim to arbitrary codimension, which is already more than we need in the paper. \par 
For general terminology in geometric measure theory we refer to \cite{Sim} and  \cite{Ilm} for more details. See also section 4 in \cite{Whi2}. Let $N$ be a compact $(n+k)$-dimensional Riemannian manifold with smooth boundary and $U \subset N$ be an open set (in our case the only relevant ambient manifolds are the closed balls $\overbar{B_R(0)}$). We denote by $\mathcal{IM}_m(U)$ the space of integral $m$-rectifiable Radon measure on $U$. For $\mu \in \mathcal{IM}_m(U)$ we let $V(\mu)$ denote the integral $m$-rectifiable varifold associated to $\mu$. Given an $n$-dimensional smooth submanifold $\Sigma$ whose boundary $\Gamma$ is a smooth manifold of $\partial N$, we define 
\begin{align*}
    \mathcal{V}_n(U,\Gamma) = \{ \mu \in \mathcal{IM}_n(U) \mid \abs{\nu} \le 1 \;\; \mathcal{H}^{n-1}\text{-a.e.}\}
\end{align*}
where $\nu$ is the normal given by the divergence theorem: for every compactly supported $C^1$ vector field on $U$,
\begin{align*}
    \int \Div_{V(\mu)} X d\mu = -\int H \cdot X d\mu + \int \nu \cdot X d(\mathcal{H}^{n-1} \llcorner \Gamma)
\end{align*}
Here $H$ is the (generalized) mean curvature vector of $V(\mu)$. \par 
By an integral $n$-Brakke flow with boundary $\Gamma$ in $U \subset N$ we mean a collection of integral Radon measures $\{\mu_t\}_{t \in I}$ for some time interval $I$ such that
\begin{enumerate}
    \item For a.e. $t \in I$, $\mu_t \in \mathcal{V}_n(U,\Gamma)$
    \item For every compact interval $[a,b] \subset I$ and compact set $K \subset U$
    \begin{align*}
        \int_a^b \int_K (1+\abs{H}^2) d\mu_t dt < \infty
    \end{align*}
    \item For $[a,b] \subset I$ and every $u \in C_c^2(U \times [a,b], \mathbb{R}^+)$,
    \begin{align*}
        \int \phi d\mu_b - \int \phi d\mu_a \le \int_a^b \int \left(-\phi \abs{H}^2 + H \cdot \nabla \phi + \frac{\partial \phi}{\partial t} \right)d\mu_t dt
    \end{align*}
\end{enumerate}
(Definition 10 of \cite{Whi2}). We will refer to the above flow simply as an integral Brakke flow with boundary if $U$ and $\Gamma$ are understood. At a smooth interior point of the flow, inequality (3) becomes equality and the usual mean curvature flow equation 
\begin{align*}
    \left(\frac{\partial x}{\partial t}\right)^\perp = H_{\Sigma}(x)
\end{align*}
is satisfied. \par
For the rest of the paper, we will restrict ourselves to the case $N \subseteq \mathbb{R}^{n+k}$, although everything in this section can be formulated for a general Riemannian manifold $N$ by an isometric embedding into some Euclidean space (see \cite{Whi2}). For $\rho > 0$ and any set $S \subseteq \mathbb{R}^{n+k}$ we write $\rho S = \{x \in \mathbb{R}^{n+k} \mid \rho^{-1} x \in S\}$. Given an integral Brakke flow $\mathcal{M} = \{\mu_t\}_{t \in [0,T]}$ in $N$ with boundary $\Gamma$ , the parabolic rescaling about $(0,0)$ of scale $\rho > 0$ is an integral Brakke flow in $\rho N$ with boundary $\rho \Gamma$ given by
\begin{align*}
    \rho \mathcal{M} = \{\rho \mu_{\rho^{-2}t}\}_{t \in [0,\rho^2 T]}
\end{align*}
Standard parabolic blow-up argument together with compactness theorem and monotonicity formula for Brakke flow with boundary implies the existence of tangent flows at any given point $(p,t) \in N \times (0,T]$, see Sections 6-8 in \cite{Whi2}. Note that the tangent flow might not be unique. \par 
As is the case for Brakke flow without boundary, one has to take into account of sudden loss of mass. We say that an integral Brakke flow with boundary $\{\mu_t\}_{t \in I}$ is unit-regular if, for $p \in N$ and $t \in I$, one of the tangent flows at $(p,t)$ is a multiplicity-1 plane or half plane, then the flow is smooth near $(p,t)$ and there is no sudden loss of mass (Definition 24 in \cite{Whi2}). We say that the flow is standard if it is unit-regular and $\partial [[\mu_t]] = [[\Gamma]]$ for $t$-a.e., where $[[\mu]]$ denotes the rectifiable mod 2 flat chain associated to $\mu \in \mathcal{IM}_n(N)$ and $[[\Gamma]] = [[\mathcal{H}^{n-1}\llcorner \Gamma]]$ (Definition 28 of \cite{Whi2}). With these terminology we can now state a simplified version of the fundamental existence theorem of White.
\begin{thm}[Theorem 1, Theorem 29 in \cite{Whi2}]
Let $N$ and $\Sigma$ be as above and suppose further $\partial N$ is mean-convex. There exists a standard $n$-integral Brakke flow $\{\mu_t\}_{t \in [0,\infty)}$ in $N$ with boundary $\Gamma$ with $\mu_0 = \mathcal{H}^n \llcorner \Sigma$.
\end{thm}
\begin{rem}
Of course, Brakke flow suffers from the lack of uniqueness. In general given an initial surface (with boundary) there could be multiple ways to evolve it by mean curvature, but the above theorem guarantees that a standard one exists. Moreover, in the codimension 1 case ($k=1$), White \cite{Whi2} proved a boundary regularity theorem which ensures boundary points along the flow are indeed regular points of the flow. This is not generally true in higher codimensions, but since we are mostly interested in the flow inside the half ball, we do not need to use any regularity assumption on the boundary. 
\end{rem}
Fix a $C^2$ cutoff function $\phi_0:\mathbb{R} \to \mathbb{R}$ that is identically 1 on $\{\abs{x} \le \frac{1}{2}\}$ and identically 0 on $\{\abs{x} > 1\}$ and let $\phi_R(x) = \phi\left(\frac{\abs{x}}{R}\right)$ for $x \in \mathbb{R}^{n+k}$. Given an integral Brakke flow with boundary $\mathcal{M} = \{\mu_t\}_{t \in [t_0,t_1]}$ in $\overbar{B_{R}(0)}$ and a point $(x_0,t) \in B_R(0) \times (t_0,t_1]$, the Gaussian density ratio of $\mathcal{M}$ at $(x_0,t)$ of scale $r$ (with $t - r^2 > t_0$) is defined similarly as 
\begin{align*}
    \Theta(\mathcal{M},(x_0,t),r) = \int \frac{\phi_R}{(4\pi r^2)^{n/2}} e^{-\frac{\abs{x-x_0}^2}{4r^2}} d\mu_{t-r^2}
\end{align*} \par 
The key observation is that, given an integral $n$-Brakke flow $\mathcal{M} = \{\mu_t\}_{t \in I}$ with boundary in $\overbar{B_R(0)}$ with the property that every point in $B_{R/2}(0) \times I$ is a smooth point of the flow, $\mathcal{M}$ restricts to a mean curvature flow without boundary in $B_{R/2}(0)$. To see this, first note that every interior point of the flow is a regular point by assumption, so the mean curvature flow equation is satisfied. Then by standard ODE theory we can choose a suitable time-varying parameterization to ensure that the definitions in 2.1 are satisfied. Moreover, the Gaussian density ratios of the restricted flow is bounded above by the Gaussian density ratios of the original flow with boundary, which are in turn bounded above by the entropy.
\section{Rigidity of Planes}
We begin with a reformulation of White's local regularity theorem \cite{Whi1} as observed by Ilmanen-Neves-Schulze \cite{INS}.
\begin{thm}
\label{t31}
Let $R \ge 1$. Suppose $\mathcal{M} = \{\mu_t\}_{t\in[0,\infty)}$ is an integral Brakke flow with boundary in $\overbar{B_R(0)}$ that is smooth in $B_{R/2}(0)$ for at least $t \in [0,1]$. If for some sufficiently small $\varepsilon > 0$, all the Gaussian density ratio at $(x_0,t) \in B_{R/4}(0) \times [0,1]$ satisfies 
\begin{align*}
    \Theta(\mathcal{M},(x_0,t),r) \le 1 + \varepsilon
\end{align*}
for $0 < r < \min\{\frac{R}{2} - \abs{x},\abs{t}^{1/2}\}$, then we have the curvature bound
\begin{align*}
    \abs{A_{\Sigma_t}(x)} \le \frac{C}{\sqrt{t}}
\end{align*}
for $x \in \supp \mu_t \cap B_{R/4}(0)$, $t \in (0,1]$, where $\Sigma_t$ is such that $\mu_t \llcorner B_{R/4}(0) = \mathcal{H}^n \llcorner \Sigma_t$ and $C = C(n,k,\varepsilon)$. Moreover, $C \to 0$ as $\varepsilon \to 0$.
\end{thm}
\begin{proof}
Consider first $R = 1$. The claim follows from White's regularity theorem \cite{Whi1} once we prove the corresponding Gaussian density ratio for the restricted mean curvature flow $\mathcal{M}' = \{\mu_t\llcorner B_{1/2}(0)\}_{t \in [0,1]}$ is also bounded above by $1+ \varepsilon$, provided $\varepsilon$ is sufficiently small. Indeed, our cutoff function $\varphi$ is chosen so that it is identically 1 on $B_{1/2}(0)$, this gives that \begin{align*}
    \Theta(\mathcal{M}',(x_0,t),r) = \int \frac{\phi_1}{(4\pi r^2)^{n/2}} e^{-\frac{-\abs{x-x_0}^2}{4r^2}} d(\mu_{t} \llcorner B_{1/2}(0)) \le \Theta(\mathcal{M},(x_0,t),r)
\end{align*}
where $r$ is as in the statement of the theorem. Finally, if $R > 1$ the theorem follows by applying the $R = 1$ case to a suitable smaller ball of radius 1.
\end{proof}
We are now ready to prove \cref{t14}.
\begin{proof}[Proof of \cref{t14}]
Let $x \in \Sigma$. For sufficiently small $R > 0$ consider $\Sigma \cap \overbar{B_R(x)}$ and let $\mathcal{M} = \{\mu_t\}_{t \in [0,\infty)}$ be the standard integral Brakke flow starting from $\mu_0 = \mathcal{H}^n \llcorner (\Sigma \cap \overbar{B_R(x)})$. By entropy bound and \cref{t31} the restricted flow $\mathcal{M} \llcorner B_{R/2}(x)$ never develops a singularity. The second fundamental form bound immediately yields that the flow is flat in $B_{R/2}(x)$. Since $x$ is arbitrary we conclude $\Sigma$ is also flat.
\end{proof}
\section{Hausdorff Stability}
In this section we prove \cref{t13}. First we show the two-sided clearing out lemma analogous to Theorem 1.2 in \cite{Wang}. We will prove a higher codimension version under a much stronger condition $\lambda[\Sigma] \le 1 + \delta$, which, in view of \cref{t14}, forces the blow up limit to be a plane.\par 
For convenience in the section we assume $B_1(0) \subseteq U$. Note also we state our results in the context of smooth mean curvature flows instead of matching motions. 
\begin{prop}[Two-sided clearing out, cf. Theorem 1.2 in \cite{Wang}]
\label{p51}
There are constants $\delta = \delta(n,k)$ and $\eta = \eta(n,k)$ such that: Suppose $\mathcal{M} = \{\Sigma_t\}_{t \in I}$ is a smooth mean curvature flow in some open set $U \subseteq \mathbb{R}^{n+k}$ with entropy $\lambda[\mathcal{M}] \le 1 + \delta$. Suppose $(t_0 - R^2,t_0+R^2)\subseteq I$ and $x_0 \in \Sigma_{t_0} \cap B_{1/2}(0)$, then for all $0 < r < \min\{R,\frac{1}{2}\}$
\begin{align*}
    \mathcal{H}^n(B_r(x_0) \cap \Sigma_{t_0 - r^2}) \ge  \eta r^n
\end{align*}
and 
\begin{align*}
    \mathcal{H}^n(B_r(x_0) \cap \Sigma_{t_0 + r^2}) \ge  \eta r^n
\end{align*}
\end{prop}

\begin{proof}
For simplicity we assume $(x_0,t_0) = (0,0)$. Suppose for a contradiction that the proposition is false. There is a sequence of smooth mean curvature flows $\mathcal{M}^i = \{\Sigma^i_t\}_{t \in (-(R^i)^2,(R^i)^2)}$ in $U^i \subseteq \mathbb{R}^{n+k}$ that satisfies $0 \in \Sigma_0$ with entropy $\lambda[\mathcal{M}] \le 1 + \frac{1}{i}$, and a sequence of radii $r^i < \min\{R^i,\frac{1}{2}\}$ such that 
\begin{align*}
    \mathcal{H}^n (B_{r^i}(0) \cap \Sigma^i_{-(r^i)^2}) < \eta (r^i)^n \text{ or } \mathcal{H}^n (B_{r^i}(0) \cap \Sigma^i_{(r^i)^2}) < \eta (r^i)^n
\end{align*}
Consider the parabolic rescaling of the flow $(r^i)^{-1}\mathcal{M}^i$, which exists in $B_2(0)$ for time $[-1,1]$. The rescaled flows satisfy
\begin{align*}
     \mathcal{H}^n (B_{1}(0) \cap \Sigma^i_{-1}) < \eta \text{ or } \mathcal{H}^n (B_{1}(0) \cap \Sigma^i_{1}) < \eta
\end{align*}
By compactness of mean curvature flows \cite{Ilm} there is a subsequence converging to a smooth mean curvature flow $\mathcal{M} = \{\Sigma_t\}_{t \in [-1,1]}$ in $B_2(0)$. Second fundamental form bound \cref{t31} gives that for sufficiently large $i$, $B_R(0) \cap \Sigma^i_t$ can be written as a graph over its tangent plane for some $R >0$. This implies that the limiting flow is nonempty. Evidently $\mathcal{M}$ has entropy $\lambda[\mathcal{M}] = 1$ and is therefore a flow of an $n$-dimensional plane $P$ passing through 0 in $\mathbb{R}^{n+k}$. But for an $n$-dimensional plane passing though 0, we clearly have 
\begin{align*}
     \mathcal{H}^n (B_{1}(0) \cap P) = \omega_n 
\end{align*}
where $\omega_n$ is the volume of the unit ball in $\mathbb{R}^n$, so if $\eta < \omega_n$ we get a contradiction.
\end{proof}

Next we prove \cref{t13}. We need to be a little bit more careful as the Hausdorff distance estimate used to prove this type of stability result does not necessarily hold for mean curvature flows in $B_1(0)$, since the flow might (eventually) go out of the ball. 
\begin{proof}[Proof of \cref{t13}]
We argue by contradiction. Suppose the theorem is false for some $\varepsilon > 0$, then we can find a sequence of $n$-dimensional smooth, properly embedded surfaces $\{\Sigma^i\}$ with $\lambda[\Sigma^i] < 1 + \frac{1}{i}$ but $\dist_P(\Sigma^i) \ge \varepsilon$. By definition this means that there are $R^i > 0$, $p^i \in \Sigma^i$ such that 
\begin{align}
    \dist_H(\Sigma^i \cap B_{R^i}(p^i), (P+p^i) \cap B_{R^i}(p^i)) > \varepsilon R^i
\end{align}
for every $n$-dimensional plane $P$. Since the Reifenberg distance is invariant under translation and dilation, we may assume $R^i = 1$ and $p^i = 0$. Let $\mathcal{M}^i = \{\mu_t^i\}_{t \in [0,\infty)}$ be the standard integral Brakke flow with boundary in $\overbar{B_4(0)}$ such that $\mu_0^i = \mathcal{H}^n \llcorner (\overbar{B_4(0)} \cap \Sigma^i)$. For sufficiently large $i$, White's local regularity theorem ensures that the restricted flow to $B_2(0)$ never develops a singularity. \par 
By compactness of mean curvature flow \cite{Ilm} there is a subsequence, still denoted by $\{\mathcal{M}^i\}$ such that $\mathcal{M}^i \llcorner B_2(0)$ converges to a smooth mean curvature flow $\mathcal{M}$ in $B_2(0)$. Again we have that the limiting flow is nonempty and that $\lambda[\mathcal{M}] = 1$ and is therefore a flow of $n$-dimensional planes by \cref{t13}. We claim that for sufficiently large $i$ we have 
\begin{align*}
    \dist_H(\Sigma^i \cap B_1(0), P_0 \cap B_1(0)) < \varepsilon
\end{align*}
which will give a contradiction to (4.1). \par 
To this end suppose first $x \in \Sigma^i \cap B_1(0)$. By \cref{p51}, the intersection $B_{\varepsilon/4}(x_0) \cap \supp \mu_{\varepsilon^2/16}^i$ is not empty for every $x_0 \in \Sigma^i$, so we can find $y \in \supp \mu_{\varepsilon^2/16}^i \cap B_{1+\varepsilon/4}(0)$ with $\abs{x-y} < \frac{\varepsilon}{4}$. Since $\mu_{\varepsilon^2/16}^i \to \mathcal{H}^n \llcorner P_0$ inside $B_2(0)$, for sufficiently large $i$ we can find $z' \in B_{1+\varepsilon/4}(0) \cap P_0$ such that $\abs{y - z} < \frac{\varepsilon}{4}$. If $z' \not \in B_1(0)$ let $z$ be the intersection of $\partial B_1(0)$ and the line connecting $0$ and $z'$, otherwise set $z = z'$. Then
\begin{align*}
    \abs{x - z} \le \abs{x-y} + \abs{y-z'} + \abs{z' - z} < \frac{3}{4}\varepsilon
\end{align*}
Similarly, given $z \in P_0 \cap B_1(0)$, if $z \not \in B_{1-\varepsilon/4}(0)$ we let $z'$ be the intersection of $\partial B_{1-\varepsilon/4}(0)$ and the line connecting 0 and $z$, otherwise we let $z' = z$. Again for sufficiently large $i$ we can find $y \in \mu_{\varepsilon^2/16}^i \cap B_{1 - \varepsilon/4}(0)$ such that $\abs{y - z} < \frac{\varepsilon}{4}$. Finally by \cref{p51} there is $x \in \supp \mu_0^i \cap B_1(0)$ such that $\abs{x-y} < \frac{\varepsilon}{4}$. Triangle inequality yields $\abs{x-z} < \frac{3}{4}\varepsilon$. This proves our claim and thus completes the proof. 
\end{proof}
\section{Asymptotic for Self-expanders} 
\label{exp}
In this section we prove \cref{t15}. Given an $n$-dimensional smooth, properly embedded self-expander $\Sigma$ in $\mathbb{R}^{n+k}$, let
\begin{align*}
    \mathcal{M}_\Sigma = \bigcup_{t > 0} \sqrt{t} \Sigma \subseteq \mathbb{R}^{n+k} \times \mathbb{R}
\end{align*}
be the spacetime track of $\Sigma$ under the mean curvature flow and let $\mathcal{C}_\Sigma = \bar{\mathcal{M}} \cap \{t = 0\}$ be the slice of the (closure of the) track at time $t = 0$. \par 
It is not hard to see that $\mathcal{C}_\Sigma$ is a (set-theoretic) cone; that is, $\mathcal{C}_\Sigma = \rho \mathcal{C}_\Sigma$ for any $\rho > 0$. Indeed, $\rho \mathcal{C}_\Sigma$ corresponds to the space-time track $\mathcal{M}_{\rho\Sigma} = \rho\mathcal{M}_\Sigma$ of $\rho \Sigma$, which is evidently the same as $\Sigma$. We observe by definition $x \in \mathcal{C}_\Sigma$ means that there is a sequence $t^i \to 0$ and $x^i \in \Sigma$ such that $\sqrt{t^i} x^i \to x$. In other words, $\sqrt{t}\Sigma$ converges pointwise to $\mathcal{C}_\Sigma$ as $t \to 0$. \par 
Since the mean curvature flow of a self-expander exists in all of $\mathbb{R}^{n+k}$, \cref{p51} immediately implies a Hausdorff distance estimate of the form 
\begin{align*}
    \dist_H(\sqrt{t_1}\Sigma,\sqrt{t_2}\Sigma) \le C\sqrt{t_1-t_2}
\end{align*}
for $t_1 > t_2 > 0$ and $C = C(n,k)$. This says that the surfaces of low entropy does not move too fast compared to the scaling, which is key to the proof of \cref{t15}. We remark that \cite{Wang} showed that the above Hausdorff distance estimate holds for hypersurfaces of entropy $\lambda[\Sigma] \le \lambda[\mathbb{S}^n]+ \delta$, and consequently all the results in this section continue to hold with the relaxed entropy bound in the case of hypersurfaces.
\begin{proof}[Proof of \cref{t15}]
Let $\delta$ be chosen as in \cref{p51}. We show the theorem with $\mathcal{C} = \mathcal{C}_\Sigma$. By compactness of the Hausdorff topology, the compact sets $\sqrt{t}\Sigma \cap \overbar{B_R(0)}$ converge uniquely as $t \to 0$ to a compact set $K \subseteq \overbar{B_R(0)}$. For simplicity, in this proof we will write 
\begin{align*}
    \dist^*_H(A,B) = \dist_H(A \cap \overbar{B_R(0)}, B \cap \overbar{B_R(0)})
\end{align*}
and similarly write $\dist^*$ for the usual distance in $\overbar{B_R(0)}$, i.e.
\begin{align*}
    \dist^*(A,B) = \dist(A \cap \overbar{B_R(0)}, B \cap \overbar{B_R(0)}) = \inf_{x \in A \cap \overbar{B_R(0)}} \dist(x,B\cap \overbar{B_R(0)})
\end{align*}
\par 
Let $\varepsilon > 0$ and take a $\varepsilon/4$-net of $\mathcal{C}_\Sigma \cap \overbar{B_R(0)}$, say 
\begin{align*}
    \mathcal{C}_\Sigma \cap \overbar{B_R(0)} \subseteq \bigcup_{j=1}^N B_{\varepsilon/4}(x^j)
\end{align*}
where $x^j \in \mathcal{C}_\Sigma \cap \overbar{B_R(0)}$. Let $x \in B_{\varepsilon/4}(x^j)$ for some $j$ and let $\{\sqrt{t^i}y^i\}$ be the corresponding sequence that converges to $x^j$. By Hausdorff distance estimate, for each $t \in (t^{i+1},t^i)$ there is $y_t \in \Sigma$ such that 
\begin{align*}
    \abs{\sqrt{t} y_t - \sqrt{t^i}y^i} \le C\sqrt{t^i - t} < \varepsilon
\end{align*}
if $t^i$ is sufficiently small. Hence by triangle inequality we have 
\begin{align*}
    \abs{\sqrt{t}y_t - x} \le \abs{x - x^j} + \abs{\sqrt{t} y_t - \sqrt{t^i}y^i} + \abs{\sqrt{t^i}y^i - x} < \varepsilon
\end{align*}
for $t^i$ sufficiently small. Take $t_0$ to be the minimum over all $j$ we see that $t < t_0$ implies $\dist^*(\mathcal{C}_\Sigma,\sqrt{t}\Sigma) < \varepsilon$. \par On the other hand given $\sqrt{t}x \in \sqrt{t}\Sigma$, we can use the Hausdorff distance estimate to define a Cauchy sequence in the following way: Set $x_1 = x$. Having defined $x_i$, let $x_{i+1} \in \Sigma$ be such that
\begin{align*}
    \abs{\frac{\sqrt{t}}{2^{i}}x_i - \frac{\sqrt{t}}{2^{i+1}}x_{i+1}} \le C\sqrt{3}\frac{\sqrt{t}}{2^{i+1}}
\end{align*}
Let $y = \lim_{i\to \infty} x_i$, then evidently $y \in \mathcal{C}_\Sigma$, and we have the distance estimate
\begin{align*}
    \abs{\sqrt{t} x - y} \le C\sqrt{3t}\sum_{i=1}^\infty \frac{1}{2^{i+1}} < \frac{\varepsilon}{2}
\end{align*}
provided $t$ is sufficiently small (but independent of $x$). This shows that $y \in B_{R+\varepsilon/2}(0)$. If $y \in B_R(0)$ then the above estimate yields immediately $\dist(\sqrt{t}x,y) < \varepsilon$. Otherwise let $\bar{y} \in \mathcal{C}_\Sigma$ be the intersection of the line connecting $0$ to $y$ and $\partial B_R(0)$ (since $\mathcal{C}_\Sigma$ is a cone), then
\begin{align*}
    \abs{\sqrt{t}x - \bar{y}} \le \abs{\sqrt{t}x - y} + \abs{y - \bar{y}} < \frac{\varepsilon}{2} + \frac{\varepsilon}{2} = \varepsilon
\end{align*}
This proves $\dist^*(\sqrt{t}\Sigma,\mathcal{C}_\Sigma) < \varepsilon$ for sufficiently small $t$. Combining the above two inequalities gives $\dist_H^*(\sqrt{t}\Sigma ,\mathcal{C}_\Sigma) < \varepsilon$. Finally triangle inequality gives 
\begin{align*}
    \dist_H^*(K,\mathcal{C}_\Sigma) \le \dist_H^*(\sqrt{t}\Sigma,\mathcal{C}_\Sigma) + \dist_H^*(\sqrt{t}\Sigma,K) < 2\varepsilon
\end{align*}
for sufficiently small $t$. This implies $K = \mathcal{C}_\Sigma$. \par 
To show the convergence rate it suffices to show $\dist^*(\mathcal{C}_\Sigma,\sqrt{t}\Sigma) \le C\sqrt{t}$ for $t$ sufficiently small. Let $y \in \mathcal{C}_\Sigma \cap \overbar{B_R(0)}$. By triangle inequality and Hausdorff distance estimates, for every integer $n$ we have 
\begin{align*}
    \dist^*(y,\sqrt{t}\Sigma) \le \dist^*\left(y,\frac{\sqrt{t}}{2}\Sigma\right) + \frac{C\sqrt{3}}{2} \sqrt{t} \le \dist^*\left(y,\frac{\sqrt{t}}{2^n}\Sigma\right) + C\sqrt{3t}\sum_{i=1}^n \frac{1}{2^i} 
\end{align*}
So for fixed $\varepsilon > 0$ there is some suffciently large $n$ such that 
\begin{align*}
    \dist^*(y,\sqrt{t}\Sigma) \le C\sqrt{3t} + \varepsilon
\end{align*}
holds. This completes the proof.
\end{proof}
On the other hand, instead of a set-theoretic limit, one could treat $\{\sqrt{t}\Sigma\}_{t > 0}$ as a sequence of  varifolds (that is, Radon measures on $\mathbb{R}^{n+k} \times G(n+k,n)$, where $G(n+k,n)$ is the space of $n$-dimensional planes in $\mathbb{R}^{n+k}$) and use compactness of Radon measures to deduce the existence of a limiting varifold $V$, which again is a Radon measure on the Grassman bundle $\mathbb{R}^{n+k} \times G(n+k,n)$. Of course, the limiting varifold depends on the subsequence one takes, so $V$ needs not to be unique. \par 
Let $\mathcal{LV}(\Sigma)$ denote the space of all possible limiting varifolds associated to the self-expander $\Sigma$, and given $V \in \mathcal{LV}(\Sigma)$ let $\mu_V$ be its associated mass measure, which is a Radon measure on $\mathbb{R}^{n+k}$ given by the explicit formula
\begin{align*}
\mu_V(U) = V(\pi^{-1}(U))
\end{align*}
for $U \in \mathbb{R}^{n+k}$, where $\pi$ is the projection of $U \times G(n+k,n)$ onto $U$. In the low entropy case, it is expected that $\mathcal{LV}(\Sigma)$ consists of a unique element which is countably $n$-rectifiable, has integral density, and is supported precisely on $\mathcal{C}_\Sigma$ such that the tangent planes agree. In other words, the set-theoretic limit should coincide with the limiting varifold. Assuming this is true, varifold convergence implies that we can interpret the cone $\mathcal{C}_\Sigma$ as the initial data of the integral Brakke flow of the self-expander $\Sigma$; that is, $\Sigma$ can be viewed as an self-expander coming out of the cone $\mathcal{C}_\Sigma$. It also implies that $\lambda[\Sigma] = \lambda[\mathcal{C}_\Sigma]$. So if $\lambda[\Sigma]$ is sufficiently close to 1, an easy generalization of \cref{t13} shows that the cone $\mathcal{C}_\Sigma$ is also Reifenberg flat in the sense of \cref{d13}.\par 
Unfortunately we cannot prove this, but we will show below that at least the mass measure associated to each $V \in \mathcal{LV}(\Sigma)$ is countably $n$-rectifiable and is supported on $\mathcal{C}_\Sigma$.
\begin{thm}
\label{t51}
Let $\delta$ be as in \cref{t15}. Suppose $\Sigma$ is a properly embedded $n$-dimensional self-expander in $\mathbb{R}^{n+k}$ with $\lambda[\Sigma] \le 1 + \delta$. Let $V \in \mathcal{LV}(\Sigma)$, then $\mu_V$ is countably $n$-rectifiable and $\supp \mu_V = \mathcal{C}_\Sigma$. 
\end{thm}
In view of Preiss' theorem \cite{Pre} (see also \cite{Del}), to deduce rectifiability it suffices to show that the area ratios of the varifold are bounded above and below. The upper bound is a direct consequence of the finite entropy assumption. To establish the lower bound, we first need a lemma that works in general for all surfaces of bounded entropy.
\begin{lem}
\label{l52}
Suppose $\Sigma$ is an $n$-dimensional smooth surface in $\mathbb{R}^{n+k}$ with bounded entropy. For every $\varepsilon > 0$ there is $R = R(n,k,\Sigma,\varepsilon)$ such that 
\begin{align*}
    \int_{\Sigma \cap (\mathbb{R}^{n+k} \setminus B_R(0))} \frac{1}{(4\pi)^{n/2}} e^{-\frac{\abs{x}^2}{4}}d\mathcal{H}^n \le \varepsilon
\end{align*}
\end{lem}
\begin{proof}
Note that bounded entropy implies (at most) polynomial volume growth, i.e.
\begin{align*}
    \mathcal{H}^n \llcorner \Sigma(B_R(x)) \le C\lambda[\Sigma]R^n
\end{align*}
for every $R > 0$ and $x \in \mathbb{R}^{n+k}$. Now since there is a dimensional constant $N = N(n,k)$ such that $B_{2R}(0) \setminus B_R(0)$ can be covered by $N$ balls of radius $R/2$, we have 
\begin{align*}
    \int_{\Sigma \cap (B_{2R}(0) \setminus B_R(0))} \frac{1}{(4\pi)^{n/2}} e^{-\frac{\abs{x}^2}{4}}d\mathcal{H}^n \le \frac{2^{-n}CN\lambda[\Sigma]}{(4\pi)^{n/2}} e^{-\frac{R^2}{16}} R^n
\end{align*}
More generally doing this in $B_{2^{i+1}R}(0) \setminus B_{2^i R}(0)$ we have 
\begin{align*}
    \int_{\Sigma \cap (B_{2^{i+1}R}(0) \setminus B_{2^iR}(0))} \frac{1}{(4\pi)^{n/2}} e^{-\frac{\abs{x}^2}{4}}d\mathcal{H}^n \le \frac{CN\lambda[\Sigma]}{(4\pi)^{n/2}} e^{-\frac{2^{2i-2} R^2}{4}} 2^{(i-1)n}R^n
\end{align*}
Summing over $i$ we arrive at
\begin{align*}
    \int_{\Sigma \cap (\mathbb{R}^{n+k} \setminus B_R(0))} \frac{1}{(4\pi)^{n/2}} e^{-\frac{\abs{x}^2}{4}}d\mathcal{H}^n \le \frac{C N \lambda[\Sigma]}{(4\pi)^{n/2}}\left(\sum_{i=0}^\infty 2^{(i-1)n} e^{-2^{2i-4}R^2} \right)R^n
\end{align*}
Let us agree that $R > 1$, and observe that $2^{2i-4} \ge i-1$ for all integers $i \ge 2$, then above bound then becomes 
\begin{align*}
    \int_{\Sigma \cap (\mathbb{R}^{n+k} \setminus B_R(0))} \frac{1}{(4\pi)^{n/2}} e^{-\frac{\abs{x}^2}{4}}d\mathcal{H}^n \le \frac{C N \lambda[\Sigma]}{(4\pi)^{n/2}}\left(2e^{-\frac{R^2}{4}} + \sum_{i=2}^\infty ((2R)^n e^{-R^2})^{i-1}\right)
\end{align*}
where we have isolated the first two terms. The quantity inside the bracket goes to 0 as $R \to \infty$, so by choosing $R$ large enough depending on the dimensions and $\Sigma$ (in fact only on the bound on $\lambda[\Sigma]$) we can guarantee 
\begin{align*}
    \int_{\Sigma \cap (\mathbb{R}^{n+k} \setminus B_R(0))} \frac{1}{(4\pi)^{n/2}} e^{-\frac{\abs{x}^2}{4}}d\mathcal{H}^n \le \varepsilon
\end{align*}
for any $\varepsilon > 0$.
\end{proof}
\begin{prop}
\label{p53}
Let $\delta$ be as in \cref{t15}. Suppose $\Sigma$ is an $n$-dimensional properly embedded smooth self-expander in $\mathbb{R}^{n+k}$ with $\lambda[\Sigma] \le 1 + \delta$. For any $y \in \mathcal{C}_\Sigma$, there is $\gamma = \gamma(n,k,\Sigma) > 0$ and $c > 0$ such that
\begin{align*}
    \frac{\mathcal{H}^n(B_{\gamma\sqrt{t}}(y) \cap \mathcal{C}_\Sigma)}{\omega_n (\gamma^2t)^{n/2}} > c > 0
\end{align*}
In particular, by taking $t \to 0$ we infer that $\Theta(\mathcal{H}^n \llcorner \mathcal{C}_\Sigma,y) > 0$ for all $y \in \mathcal{C}_\Sigma$.
\end{prop}
\begin{proof}
Let $\mathcal{M}$ denote the mean curvature flow associated to $\Sigma$ and let $y \in \mathcal{C}_\Sigma$. Consider the Gaussian density ratio 
\begin{align*}
    \lim_{r \to \sqrt{t}}\Theta(\mathcal{M},(y,t),r) &= \lim_{r \to \sqrt{t}}\int_{\sqrt{t-r^2}\Sigma} \frac{1}{(4\pi r^2)^{n/2}} e^{-\frac{\abs{x - y}^2}{4r^2}}d\mathcal{H}^n \\
    &= \int_{\mathcal{C}_\Sigma} \frac{1}{(4\pi t)^{n/2}} e^{-\frac{\abs{x- y}^2}{4t}}d\mathcal{H}^n
\end{align*}
which is at least 1 by the entropy bound. Observe after a change of variables the integral is 
\begin{align*}
    \int_{(\mathcal{C}_\Sigma - \frac{y}{\sqrt{t}}) \cap B_{\gamma}(0)} \frac{1}{(4\pi)^{n/2}} e^{-\frac{\abs{x}^2}{4}}d\mathcal{H}^n
\end{align*} 
Since $\mathcal{C}_\Sigma$ has bounded entropy, by \cref{l52} there is some constant $\gamma$ depending on $\Sigma$ and the dimensions such that 
\begin{align*}
    \int_{\mathcal{C}_\Sigma \cap B_{\gamma\sqrt{t}}(y)} \frac{1}{(4\pi t)^{n/2}} e^{-\frac{\abs{x- y}^2}{4t}}d\mathcal{H}^n \ge c
\end{align*}
The proposition follows since 
\begin{align*}
    \int_{\mathcal{C}_\Sigma \cap B_{\gamma\sqrt{t}}(y)} \frac{1}{(4\pi t)^{n/2}} e^{-\frac{\abs{x- y}^2}{4t}}d\mathcal{H}^n \le C\frac{\mathcal{H}^n(B_{\gamma\sqrt{t}}(y) \cap \mathcal{C}_\Sigma)}{\omega_n (\gamma^2 t)^{n/2}}  &\qedhere
\end{align*}
\end{proof}
\cref{p53} also guarantees that the dimension of the cone does not drop. This is not the case in general: if we form a similar cone for a self-shrinker, it is possible that the cone does not have the same dimension as the original surface. For example the cone formed by the cylinder $\mathbb{S}^{n-1} \times \mathbb{R}$ in $\mathbb{R}^n$ is a line (consequently there is no area ratio lower bound). \par 
\begin{proof}[Proof of \cref{t51}]
By \cref{p53}, $\Theta(\mathcal{H}^n \llcorner \mathcal{C}_\Sigma,y) > 0$ for all $y \in \mathcal{C}_\Sigma$. By Priess' Theorem \cite{Pre} it remains to show that $\supp \mu_\Sigma = \mathcal{C}_\Sigma$. \par 
It is clear from \cref{p53} that $\mathcal{C}_\Sigma \subseteq \supp \mu_\Sigma$. So suppose $x \in \mathbb{R}^{n+k} \setminus \mathcal{C}_\Sigma$. Clearly if the flow does not reach the point $x$ at any time $t > 0$ it cannot be in the support. Hence we may assume $x \in \sqrt{t_0}\Sigma$ for some $t_0 > 0$. Let $d = \dist(x,\mathcal{C}_\Sigma) > 0$ (note that $d < \abs{x}$ since $0 \in \mathcal{C}_\Sigma$). By Hausdorff distance estimate, for sufficiently small $t$ we have 
\begin{align*}
    \dist_H(\sqrt{t}\Sigma \cap \overbar{B_{2\abs{x}}(0)},\mathcal{C}_\Sigma \cap \overbar{B_{2\abs{x}}(0)}) < \frac{d}{2}
\end{align*}
This shows that $x \not \in \supp (\mathcal{H}^n \llcorner \sqrt{t}\Sigma)$ for $t$ sufficiently small, so $x$ cannot be in the support of $\mu_\Sigma$. Hence $\supp \mu_\Sigma = \mathcal{C}_\Sigma$ and we are done.
\end{proof}
\begin{rem}
Preiss' theorem seems like an overkill for the proof of rectifiability of the (limiting) cone in the low entropy case. We expect that there is a simpler proof using directly the low entropy condition of the expander and involving less geometric measure theory tools.
\end{rem}

%\bibliographystyle{alpha}
%\bibliography{ref}

\end{document}